\documentclass{amsart}
\usepackage{amsfonts}
\usepackage{amsmath,amssymb}	
\usepackage{amsthm}
\usepackage{amscd}
\usepackage{graphics}
\usepackage{graphicx}

{
\theoremstyle{definition}
\newtheorem{Def}{Definition}
\newtheorem{Ex}{Example}
\newtheorem{Rem}{Remark}

\newtheorem*{MainProb}{Main Problem}
}

\newtheorem{Cor}{Corollary}
\newtheorem{Prop}{Proposition}
\newtheorem{Thm}{Theorem}

\newtheorem{Fact}{Fact}

\begin{document}
\title[Reeb spaces and structures on preimages of regular values]{Investigating additional structures on preimages of regular values of smooth maps induced from the manifolds via the Reeb spaces and applications}

\author{Naoki Kitazawa}
\keywords{Singularities of differentiable maps. Reeb spaces. Spin structures. Framed manifolds. Cohomology classes. Differential topology of manifolds. \\
\indent {\it \textup{2020} Mathematics Subject Classification}: Primary~57R45. Secondary~57R19.}
\address{Institute of Mathematics for Industry, Kyushu University, 744 Motooka, Nishi-ku Fukuoka 819-0395, Japan\\
 TEL (Office): +81-92-802-4402 \\
 FAX (Office): +81-92-802-4405 \\
}
\email{n-kitazawa@imi.kyushu-u.ac.jp}
\urladdr{https://naokikitazawa.github.io/NaokiKitazawa.html}
\maketitle
\begin{abstract}
The {\it Reeb space} of a generic map is the space of all connected components of preimages of the map. Reeb spaces are fundamental and useful tools in the theory of Morse functions and higher dimensional variants and their applications to geometry. These spaces are polyhedra compatible with 
natural simplicial structures of the manifolds in considerable cases. For example, for Morse functions, {\it fold} maps, and proper {\it stable} maps, which are useful and fundamental tools in these studies, they are polyhedra whose dimensions are equal to those of the manifolds of the targets.
 
In the 2010s, Hiratuka and Saeki found that the top-dimensional homology group of the Reeb space does not vanish for smooth maps such as proper stable maps having some preimages containing a component which is not (oriented) null-cobordant and which has no singular points. Recently the author found some extended versions of this fact formulating and using cobordism-like groups of equivalence classes of smooth and closed manifolds. 
 
The present paper concerns cases where manifolds and preimages with no singular points of maps may have additional algebraic topological or differential topological structures other than differentiable structures and orientations. We show
a theorem for this case. We also give some explicit applications of the result where the preimages have spin or spin$^c$ structures induced from the manifolds of the domains and structures of framed manifolds for example.
As important applications, we present new explicit construction of spin cobordisms and spin$^c$ cobordisms for example.
\end{abstract}

\section{Introduction.}
\label{sec:1}
Singularity theory and geometric theory of {\it generic} smooth maps such as {\it stable} maps and applications to geometric studies of smooth manifolds have developed since the former half of the 1900s. Later we will introduce some terminologies on smooth maps we need. See \cite{golubitskyguillemin} for generic smooth maps including stable maps and their singular points for more systematic explanations. Note also that this sophisticated theory contains fundamental theory of Morse functions as a very fundamental study.

In the present paper, we study about a result on a relation between homology groups of the {\it Reeb spaces} of such maps and classes or types of preimages of regular values by Saeki and Hiratuka \cite{hiratukasaeki} and its generalization by the author \cite{kitazawa3}. The {\it Reeb space} of a continuous map is the space consisting of all connected components of  preimages of the map. Reeb spaces are fundamental and important tools in the studies, representing the manifold compactly. See \cite{reeb} as a pioneering study for example, We define the Reeb space of a map in section \ref{sec:3}.

We introduce definitions, fundamental properties and results related to generic maps and Reeb spaces.

For a polyhedron, its triangulation is taken as one giving the structure of the polyhedron. A topological space homeomorphic to a $1$ or $2$-dimensional polyhedron is known to have the unique structure of a polyhedron.
A differentiable manifold is known to be a polyhedron defined in a canonical way uniquely. A $k$-dimensional topological manifold is known to have the structure of a polyhedron and that of a differentiable manifold uniquely for $k=1,2,3$.
These polyhedra of differentiable manifolds form a subclass of the class of PL manifolds. If a PL manifold $X$ is oriented, then each simplex of degree $\dim X$ is oriented in a canonical way.

On uniqueness of the structures of polyhedra and differentiable manifolds for low-dimensional cases, see \cite{moise} for example.

The PL category is known to be equivalent to the piecewise smooth category. These categories are important in the present paper. We use terminologies and notions of the PL category in the present paper where it may be better to adopt terminologies and notions in the piecewise smooth category.

\begin{Def}
\label{def:1}
Let $X$ and $Y$ be polyhedra. A continuous map $c:X \rightarrow Y$ is said to be {\it triangulable} if there exists a pair of triangulations of $X$ and $Y$ and homeomorphisms $({\phi}_X,{\phi}_Y)$ onto $X$ and $Y$ respectively such that the composition ${{\phi}_Y}^{-1} \circ c \circ {{\phi}_X}$ is a simplicial map with respect to the given triangulations. We also say that $c$ is triangulable {\it with respect to $({\phi}_X,{\phi}_Y)$}.
\end{Def}

We give an explanation on terminologies and notation on manifolds, maps between manifolds and bundles whose fibers are manifolds. 

${\mathbb{R}}^k$ denotes the $k$-dimensional Euclidean space where ${\mathbb{R}}^1$ means the line $\mathbb{R}$.
This is naturally regarded as a smooth manifold and endowed with a standard metric. $||x|| \geq 0$ denotes the distance between $x \in {\mathbb{R}}^k$ and the origin $0$ or the value of the standard Euclidean norm at $x$.
$\mathbb{Z} \subset \mathbb{R}$ denotes the integer ring. This consists of all integers.
$S^k:=\{x \in {\mathbb{R}}^{k+1} \mid ||x||=1.\}$ is the $k$-dimensional unit sphere for an integer $k \geq 0$. This is a $k$-dimensional smooth closed and connected submanifold with no boundary.
$D^k:=\{x \in {\mathbb{R}}^k \mid ||x|| \leq 1.\}$ is the $k$-dimensional unit disk for an integer $k \geq 1$. This is a $k$-dimensional smooth compact and connected manifold.

A manifold homeomorphic to ${\mathbb{R}}^k$, $S^k$ or $D^k$ may not be diffeomorphic to this. However, this is diffeomorphic to ${\mathbb{R}}^k$ for $k \neq 4$ in the case ${\mathbb{R}}^k$, diffeomorphic to $S^k$ for $k=1,2,3,5,6$ in the case $S^k$ and diffeomorphic to $D^k$ for $k \neq 4,5$ in the case $D^k$. In the previous cases, the structures of polyhedra are known to be unique and independent of the differentiable structures.

A {\it singular point} of a differentiable map $c:X \rightarrow Y$ is a point $p \in X$ at which the rank of the differential ${dc}_p$ drops or the dimension of the image of the differential is smaller than $\min\{\dim X, \dim Y\}$.
The {\it singular set} $S(c)$ of a smooth map is defined as the set of all singular points of the map and the {\it singular value set} of the map is defined as the image of the singular set. The {\it regular value set} of the map is the complementary set of the singular value set.

A {\it diffeomorphism} on a smooth manifold means a $C^{\infty}$ homeomorphism on the smooth manifold. The {\it diffeomorphism group} of the manifold is the group of all diffeomorphisms.

We introduce fundamental notions on bundles assuming fundamental notions such as {\it base spaces}, {\it fibers}, {\it structure groups} and maps (morphisms) between bundles with same structure groups. We omit more precise explanations on fundamental notions and properties on bundles. See \cite{milnorstasheff} and \cite{steenrod} for example. 

A bundle whose fiber is a smooth manifold is a {\it smooth} bundle if the structure group is (a subgroup of) the diffeomorphism group.
A bundle whose fiber is a polyhedron is a {\it PL} bundle if the structure group consists of PL homeomorphisms.
A smooth bundle is said to be a {\it linear} bundle if the fiber is a Euclidean space, a unit sphere or a unit disk and the structure group acts on the fiber linearly in a canonical way. If the fiber is diffeomorphic to ${\mathbb{R}}^k$, then it is regarded as a {\it real vector bundle} and the dimension is $k$. 
Moreover, in some specific case, we can naturally regard this as a {\it complex vector bundle} of dimension $\frac{k}{2}$ where $k$ must be even. 

The tangent bundle of a smooth manifold $X$ is denoted by $TX$. 
This is regarded as a real vector bundle of dimension $\dim X$ whose base space is $X$. 

Sections of vector bundles are said to be {\it independent} if at each point in the base space, the values are linearly independent.

A real vector bundle or more generally, a linear bundle whose structure group is the $k$-th orthgonal group is {\it orientable} if the structure group is reduced to the $k$-th rotation group. This bundle can be {\it oriented} in two ways by regarding as linear bundles whose structure groups are the rotation group. These oriented bundles may be mutually {\it equivalent} as bundles. Here two isomorphic bundles over a same base space with a fixed structure group is {\it equivalent} if some isomorphism between the two bundles induce the identity map on the base space.

For a bundle $B$ over $X$, the restriction of the bundle to $Y \subset X$ is denoted by $B {\mid}_{Y}$.

A {\it proper} map is a map between topological spaces such that the preimage of a compact subspace is always compact.

The following fact is a fact on {\it stable} maps and the class of {\it Thom} maps, which is wider than the class of stable maps. We will give an explanation on stable maps in the next section again. 

\begin{Fact}[\cite{shiota}]
\label{fact:1}
{\rm (}Proper{\rm )} stable maps and more generally, proper Thom maps are always triangulable {\rm (}with respect to pairs of homeomorphisms giving the canonical triangulations of the smooth manifolds{\rm )}.
\end{Fact}

\begin{Fact}[\cite{hiratukasaeki} and \cite{hiratukasaeki2}]
\label{fact:2}
For a triangulable map $c:X \rightarrow Y$ with respect to $(\phi_X,\phi_Y)$, the Reeb space $W_c$ is a polyhedron given by a homeomorphism ${\phi}_c$ from a polyhedron and two maps $q_c:X \rightarrow W_c$  and $\bar{c}:W_c \rightarrow Y$ are triangulable maps with respect to the corresponding pairs of the homeomorphisms.
\end{Fact}

For triangulable smooth maps, let the corresponding triangulations of the manifolds be the ones giving the structures of the underlying PL manifolds unless otherwise stated. The following result is a key fact in the present paper. For an integer $k>0$, the $k$-dimensional smooth oriented cobordism group is denoted by ${\Omega}_k$.

\begin{Fact}[\cite{hiratukasaeki2}]
\label{fact:3}
Let $m>n$ be positive integers. Let $M$ be a closed and smooth manifold of dimension $m$ and $N$ a smooth manifold of dimension $n$ with no boundary. Let $f:M \rightarrow N$ be a triangulable and smooth map. {\rm (}Let the manifolds $M$ and $N$ be oriented.{\rm )} If for a regular value $a$, there exists a connected component of $f^{-1}(a)$ which is not {\rm (}resp. oriented{\rm )} null-cobordant, then for the Reeb space $W_f$, the homology group $H_n(W_f;\mathbb{Z}/2\mathbb{Z})$ {\rm (}resp. $H_n(W_f;{\Omega}_{m-n})${\rm )} is not trivial. 
\end{Fact}
The author extended the result in \cite{kitazawa}, generalizing the coefficient modules as groups consisting of cobordism-like equivalence classes of (oriented) differentiable manifolds. The present paper also concerns discovering variants of Fact \ref{fact:3}. As a new study, we try to find applications to geometry of differentiable manifolds and their submanifolds which may be equipped with additional structures. We concentrate on spin and spin$^c$ structures and framed submanifolds. We present our Main Problem.

\begin{MainProb}
Find variants of Fact \ref{fact:3}. Mainly, investigate restrictions on preimages of regular values with additional algebraic topological or differential topological structures, parametrized by elements of a module module over a suitable coefficient ring $R$ associated with the manifold $M$ and {\rm (}$m-n${\rm )}-dimensional smooth closed submanifolds with no boundaries. Then find applications.
\end{MainProb}


The present paper is organized as follows.

First, we review definitions and fundamental properties for Morse functions, {\it fold} maps and {\it stable} maps.
Next, we introduce {\it pseudo quotient spaces} and {\it pseudo quotient} maps, which are regarded as
generalizations of Reeb spaces and maps onto the Reeb spaces canonically obtained from given smooth maps. Note that {\it pseudo quotient spaces} were first introduced by Kobayashi and Saeki in 1996 to study algebraic topological and differential topological properties of stable maps into the plane on closed manifolds whose dimensions are greater than $2$
(\cite{kobayashisaeki}).

Last, we introduce a main result as Theorem \ref{thm:1}. This gives a new extension of Fact \ref{fact:3}. We present applications as Theorems \ref{thm:2.1}--\ref{thm:3} and they are also main results and answers to Main Problem. 
  
Hereafter, $M$ is a smooth and closed manifold of dimension $m \geq 1$, $N$ is a smooth manifold of dimension $n$ with no boundary satisfying the relation $m \geq n \geq 1$ and $f$ is a smooth map from $M$ into $N$.

\section{Morse functions, fold maps and stable maps.}
\label{sec:2}
{\it Stable} maps are essential smooth maps in higher dimensional versions of the theory of Morse functions. As simplest examples, Morse functions and their higher dimensional versions, {\it fold} maps, are stable if and only if their restrictions to the singular sets are in some senses {\it generic}. These singular sets are closed submanifolds with no boundaries as Proposition \ref{prop:1} explains. We omit rigorous explanations on {\it generic} smooth maps. A Morse function is stable if and only if at each distinct singular points, the values are distinct. Stable Morse functions exist densely and stable maps exist densely on a smooth and closed manifold if the pair of the dimensions of the manifold of the domain and the target is {\it nice}. If for the pair of dimensions, both dimensions are smaller than $7$, then the pair is nice for example. Here we regard the space of all smooth maps between two smooth manifolds without boundaries as the space endowed with the so-called {\it Whitney $C^{\infty}$ topology}. See \cite{golubitskyguillemin} for such theory.

We introduce the definition and a fundamental property of a {\it fold} map. See also \cite{saeki} for example for introductory facts and advanced studies on fold maps. 
A  {\it fold} map from an $m$-dimensional manifold with no boundary into an $n$-dimensional manifold with no boundary with $m>n \geq 1$ is a smooth map at each singular point $p$ of which it is represented by the form $$(x_1,\cdots,x_m) \mapsto (x_1,\cdots,x_{n-1},\sum_{k=n}^{m-i(p)}{x_k}^2-\sum_{k=m-i(p)+1}^{m}{x_k}^2)$$ for suitable coordinates and a suitable integer $i(p)$ satisfying $0 \leq i(p) \leq \frac{m-n+1}{2}$.

\begin{Prop}
\label{prop:1}
The integer $i(p)$ is taken as a non-negative integer smaller than or equal to $\frac{m-n+1}{2}$ uniquely and we call $i(p)$ the {\rm index} of $p$. The set of all singular points of a fixed index is a smooth closed submanifold of dimension $n-1$ with no boundary and the restriction of the fold map to the singular set is a smooth immersion as introduced before.  
\end{Prop}

If $i(p)=0$ holds for each singular point above, then the fold map is said to be a {\it special generic} map. As simplest examples, Morse functions on homotopy spheres with exactly two singular points are special generic. For precise information on special generic maps, see \cite{saeki2} for example.

A fold map into the plane is known to be stable if and only if the restriction to the singular set is a smooth immersion such that {\it crossings} are {\it normal}. A {\it crossing} of the immersion is a point in the image of the immersion such that the preimage has at least two points. For a smooth immersion of a $1$-dimensional manifold with no boundary into the plane, a crossing is said to be {\it normal} if the preimage consists of exactly two points and the sum of the images of the differentials at these two points coincides with the tangent vector space at the crossing.  

\section{Pseudo quotient spaces and pseudo quotient maps as generalizations of Reeb spaces.}
\label{sec:3}
For two topological spaces $X$ and $Y$ and a map $c:X \rightarrow Y$, ${\sim}_c$ denotes the relation on $X$ defined by the following rule and it is an equivalence relation: $x_1 {\sim}_c x_2$ holds if and  only if there exists $y \in Y$ and a connected component of $f^{-1}(y)$ containing both $x_1$ and $x_2$. The {\it Reeb space} $W_c$ of the map $c$ is the quotient space $X/{\sim}_c$. $q_c:X \rightarrow W_c$ denotes the quotient map and $\bar{c}:W_c \rightarrow Y$ denotes the map uniquely defined so that $c=\bar{c} \circ q_c$ holds.

We introduce fundamental facts on Reeb spaces of fold maps including special generic maps.
\begin{Prop}
\label{prop:2}
Stable fold maps are triangulable and the dimensions of Reeb spaces coincide with the dimensions of the manifolds of the targets.
\end{Prop}

This is a specific case of Fact \ref{fact:1}.

\begin{Prop}[\cite{saeki2}.]
\label{prop:3}
Let $m>n \geq 1$ be integers.
\begin{enumerate}
\item For any special generic map $f$ from an $m$-dimensional closed and connected manifold $M$ into an $n$-dimensional connected and non-compact manifold $N$ with no boundary, the Reeb space $W_f$ is an $n$-dimensional connected and compact manifold we can smoothly immerse into the manifold $N$ by $\bar{f}:W_f \rightarrow N$. 
\begin{enumerate}
\item The preimage of each regular value is a disjoint union of finitely many copies of the {\rm (}$m-n${\rm )}-dimensional unit sphere $S^{m-n}$.
\item $q_f(S(f))$ and the boundary $\partial W_f$ of $W_f$ coincide. The composition of the restriction of $q_f$ to the preimage of a small collar neighborhood $N(\partial W_f)$ with the canonical projection to $\partial W_f$ gives a linear bundle whose fiber is diffeomorphic to $D^{m-n+1}$.
\item The restriction of $q_f$ to the preimage of $W_f-{\rm Int}\ N(\partial W_f)$ gives a smooth bundle whose fiber is diffeomorphic to $S^{m-n}$. 
\end{enumerate}
\item Let $N$ be an $n$-dimensional connected and non-compact smooth manifold with no boundary.
Let $W_N$ be an $n$-dimensional compact and connected smooth manifold and ${\bar{f}}_N:W_N \rightarrow N$ a smooth immersion. We can construct a special generic map $f^{\prime}:M^{\prime} \rightarrow N$ on a suitable $m$-dimensional closed and connected manifold satisfying $W_N=W_{f^{\prime}}$ and $\bar{f^{\prime}}={\bar{f}}_N$. 
\end{enumerate}
\end{Prop}
\begin{Def}
\label{def:2}
Let $\mathcal{C}$ be a class of triangulable smooth maps.
Let $W_{q}$ be an $n$-dimensional polyhedron. 
Let $f_q:M \rightarrow W_{q}$ be a surjective triangulable map with respect to the pair of canonical triangulations of $M$ and $W_{q}$. $W_{q}$ is said to be a {\it pseudo quotient space} of the class $\mathcal{C}$ and $f_q$ is said to be a {\it pseudo quotient map} of the class if for each point $p \in W_{q}$ and the interior of some small
connected and closed neighborhood $N(p)$ being $n$-dimensional and satisfying $p \in {\rm Int} N(p)$, 
there exist the following objects.
\begin{enumerate}
\item A triangulable smooth map $f_{M,p}:M_p \rightarrow N_p$ of the class $\mathcal{C}$ from an $m$-dimensional closed manifold $M_p$ into an $n$-dimensional manifold $N_p$ with no boundary.
\item A point $p^{\prime} \in W_{f_{M,p}}$.
\item A small connected and closed neighborhood $N({p}^{\prime}) \subset W_{f_{M,p}}$, whose interior contains $p^{\prime} \in W_{f_{M,p}}$.
\item A pair of a diffeomorphism ${\Phi}_p$ and a PL homeomorphism ${\phi}_{p^{\prime}}$ such that 
for the maps $f_{q,N(p)}:=f_q {\mid}_{{f_q}^{-1}(N(p))}:{f_q}^{-1}(N(p)) \rightarrow N(p)$ and $q_{f,N({p}^{\prime})}:=q_{f_{M,p}} {\mid}_{{q_{f_{M,p}}}^{-1}(N({p}^{\prime}))}:{q_{f_{M,p}}}^{-1}(N({p}^{\prime})) \rightarrow N({p}^{\prime})$, 
 the relation $q_{f,N({p}^{\prime})} \circ {\Phi}_p ={\phi}_{p^{\prime}} \circ f_{q,N(p)}$ holds.
\end{enumerate} 
\end{Def}

We can define notions such as a {\it singular point}, a {\it regular value}, a {\it singular value}, the {\it singular set}, the {\it singular value set} and the {\it regular value set} for the pseudo quotient map in a natural way. 
\begin{Rem}
Kobayashi and Saeki introduced {\it pseudo quotient maps} and {\it pseudo quotient spaces} as specific versions of ones here in \cite{kobayashisaeki}. They set $\mathcal{C}$ as the class of stable maps into the plane on closed manifolds whose dimensions are greater than $2$. They studied algebraic topological and differential topological properties of stable maps into the plane including transformations of maps preserving the topologies and the differentiable structures of the closed manifolds. 
\end{Rem}

\begin{Ex}
\label{ex:1}
\begin{enumerate}
\item
\label{ex:1.1}
In Proposition \ref{prop:3}, for a suitable case, there exists another compact and smooth manifold $M_W$ whose
boundary is the original manifold $M$ of the domain of the special generic map $f$ and which collapses to the compact manifold $W_f$ or the Reeb space. Furthemore, we can obtain the manifold bounded by the original manifold so that each connected component of each preimage of each regular value is bounded by a copy of a unit disk embedded smoothly and properly in this new compact manifold and that the disks bounded by distinct connected
components of preimages of regular values do not intersect. The constructed compact manifold $M_W$ is the total space of a smooth bundle over the given compact manifold $W_f$ whose fiber is diffeomorphic to the {\rm (}$m-n${\rm )}-dimensional unit disk. We can take a smooth map $r:M_W \rightarrow W_f$ giving a collapsing, making $M_W$ the smooth bundle over $W_f$ whose fiber is diffeomorphic to $D^{m-n+1}$ and satisfying $r {\mid}_{\partial M_W}=q_f:\partial M_W=M \rightarrow W_f$.
 
We can construct the compact manifold whose
boundary is the original manifold $M$ in the topology category and the PL or piecewise-smooth category for any case.
If we need to argue in the PL category, then the bundle $M_W$ is considered as a PL bundle. 
If the pseudo quotient map of this class is a pseudo quotient map of the class of special generic maps such that the pairs $(m,n)$ of the dimensions of the domains and the targets satisfy $m-n=1,2,3$, then we can argue in the smooth category as before. For this condition on the dimensions see \cite{saeki2} for example.
\item
\label{ex:1.2}
A {\it shadow} $P$ of a $3$-dimensional closed orientable manifold $M$, introduced by Turaev, is regarded as a pseudo quotient space of the class of stable fold maps on $3$-dimensional closed and orientable manifolds into the plane. It is a $2$-dimensional polyhedron
and locally (the canonically defined quotient map of) stable fold maps into the plane
with no connected components of preimages containing two singular points of a certain type. Such a connected component is called a {\it {II}$^3$ type fiber}. 
Consult \cite{saeki2} and \cite{saeki3} for example.

Moreover, a shadow is divided into finitely many connected compact surfaces regarded as the base spaces of canonically defined smooth bundles with fibers diffeomorphic to a circle and consisting
of regular values, by a small regular neighborhood of the singular value set. For the finitely many connected compact surfaces regarded as the base spaces of the bundles and consisting
of regular values, we can assign a {\it gleam} or a rational
number represented as the product of $\frac{1}{2}$ and an integer
as information on the identification on the boundaries or the restriction of the total space of the bundle with fibers diffeomorphic to a circle to the boundary. As the previous example,
there exists a compact and smooth manifold $M_W$ whose
boundary is the original manifold $M$ and which collapses to a $2$-dimensional polyhedron $P^{\prime}$ simple homotopy equivalent to the original $2$-dimensional polyhedron $P$. Furthermore, we can obtain this manifold so that each connected component of each
preimage of each regular value bounds a copy of the $2$-dimensional unit disk embedded properly and smoothly in the compact manifold and that the disks bounded by distinct connected
components of preimages of regular values do not intersect. As in (\ref{ex:1.1}) we can also take a smooth map $r_{M,P}:M_W \rightarrow P^{\prime}$ giving a collapsing.

Note that the two homotopy equivalent $2$-dimensional polyhedra $P$ and $P^{\prime}$ are PL homeomorphic if the class of stable fold maps on closed and orientable $3$-dimensional manifolds into the plane is restricted to the class of ones such that the restrictions of the quotient maps onto the Reeb spaces to the singular sets are injective. This class is defined in Definition \ref{def:4}.

If for a point in the polyhedron of the target and a local stable fold map in Definition \ref{def:2}, we need a connected component of a preimage containing two singular points, then we can take a PL map $r_P:P^{\prime} \rightarrow P$ so that on a suitable $2$-dimensional subpolyhedron, this is a PL homeomorphism onto the complementary set of the interior of a suitable small regular neighborhood of the set of all points satisfying this property on the preimages in $P^{\prime}$. We can also obtain $r_{M,P}:M_W \rightarrow P^{\prime}$ and $r_P$ so that the restriction of $r_P \circ r_{M,P}:M_W \rightarrow P$ to the original manifold $M$ is the original pseudo quotient map. For the map $r_P \circ r_{M,P}$, the preimages of compact submanifolds in the regular value set of the original map are $4$-dimensional compact submanifolds of the constructed compact manifold $M_W$. This map makes the $4$-dimensional compact submanifolds the total spaces of smooth bundles over the submanifolds in $P$ whose fibers are diffeomorphic to unit disks. If we need to argue in the PL or the piecewise smooth category, then bundles are considered as PL bundles.

For more precise studies on geometric theory of shadows, see \cite{costantino}, \cite{costantinothurston}, \cite{ishikawakoda} and \cite{turaev} for example. For the compact manifold $M_W$ bounded by the original manifold $M$ obtained by observing the map and polyhedra playing same roles as $P$ and $P^{\prime}$ play before,
see also \cite{saeki3} and \cite{saekisuzuoka} for example. In these two articles similar situations are considered and similar results are discussed under the situations that the pairs of dimensions are general. In these discussions, to obtain the compact and smooth manifold $M_W$ whose boundary is the original manifold $M$ of the domain of the given smooth map and argue in the smooth category, we need to assume the pair of the dimensions of the domain and the target as $(2,1), (3,1), (4,1), (3,2)$ and $(4,2)$ for example. In cases where the dimensions of manifolds of the domains and the targets are general, we can argue in the PL category and the topological category. See \cite{kitazawa2}, concentrating on {\it simple-standard-spherical} fold maps in Definition \ref{def:4} later on this argument, for example.

We will give a related exposition for Theorem \ref{thm:2.2} later.
\end{enumerate}
\end{Ex}

A closed interval $I$ smoothly embedded in a manifold $Y$ with no boundary is said to be {\it generic} with respect to the image of a smooth immersion $i:X \rightarrow Y$ of codimension $1$ of a manifold with no boundary if the following properties hold.
\begin{enumerate}
\item The closed interval $I$ contains finitely many points of the image of the immersion $i$ and these points are in the interior of the interval.
\item For each point $p \in {\rm Int}\ I$ as before, for each point $q$ of the preimage $i^{-1}(p)$ of the immersion, the sum of the tangent vector space of $I$ at $p \in I$ and the image of the differential of $i$ at $q$ coincides with the tangent vector space of $Y$ at $p \in Y$.
\end{enumerate}
FIGURE \ref{fig:1} and FIGURE \ref{fig:2} show the preimage of a generic closed interval with respect to the singular value set of a stable fold map and containing exactly one singular value in the interior
in Example \ref{ex:1} (\ref{ex:1.2}). Dotted segments are for generic curves and black lines are for the singular value sets.
\begin{figure}
\includegraphics[width=30mm]{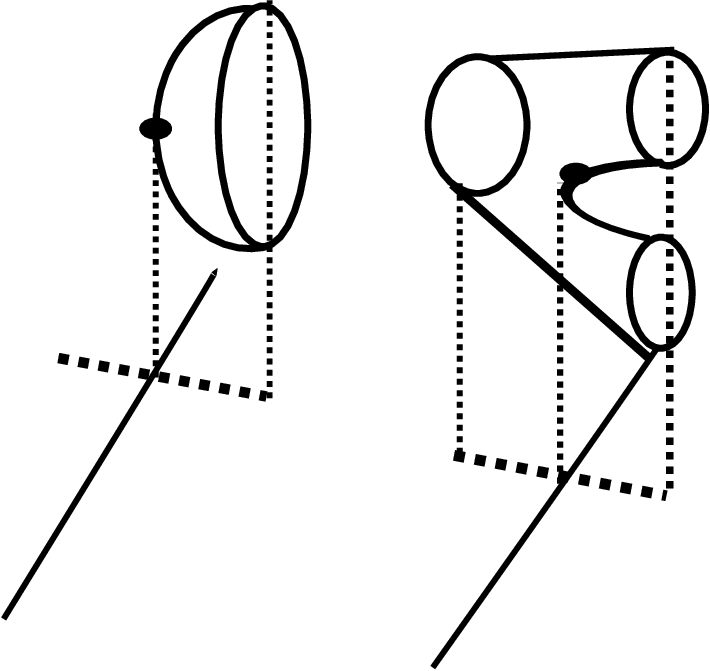}
\caption{The preimage of a sufficiently small generic closed interval around a singular value whose preimage is connected and contains exactly one singular point in Example \ref{ex:1} (\ref{ex:1.2}).}
\label{fig:1}
\end{figure}
\begin{figure}
\includegraphics[width=30mm]{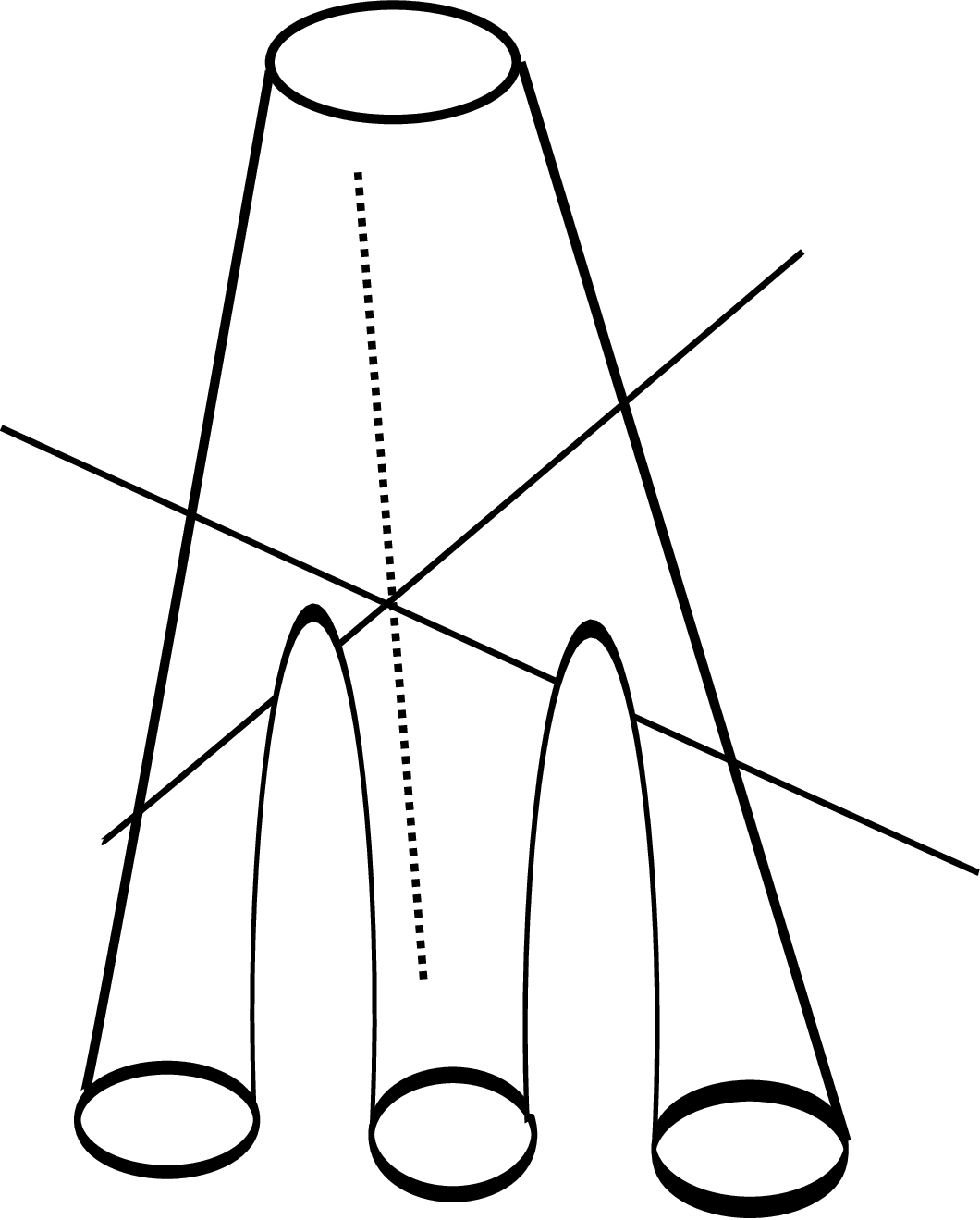}
\caption{The preimage of a sufficiently small generic closed interval around a singular value whose preimage is connected and contains exactly two singular points in Example \ref{ex:1} (\ref{ex:1.2}).}
\label{fig:2}
\end{figure}

\section{Main Theorems --Manifolds with additional structures, induced structures on preimages of regular values and top-dimensional homology groups of Reeb spaces or pseudo quotient spaces--.}
\label{sec:4}
We present main theorems. 

Before that, we review and introduce notions on homology classes and orientations of PL or smooth manifolds.

A {\it PID} means a so-called {\it principal ideal domain} having the unique identity element $1$ different from the zero element $0$. A closed and connected PL or smooth manifold $X$ is known to be orientable if $H_{\dim X}(X;A)$ is isomorphic to $A$ for any PID. 

If a PL or smooth manifold $X$ is orientable and oriented, then each simplex of degree $\dim X$ is oriented in a canonical way. For the oriented manifold $X$, the sum of all these oriented simplices is taken (in the chain complex whose coefficient ring is $A$). If the manifold $X$ is also closed and connected, then the homology class the sum represents is said to be the {\it fundamental class} of the (oriented) manifold $X$.
For a PL (smooth) manifold $X$, a class $c \in H_k(X;A)$ is {\it represented} by a PL (resp. smooth) closed submanifold $Y \subset {\rm Int}\ X \subset X$ with no boundary which is also orientable if the value of the homomorphism induced canonically from the inclusion $i:Y \rightarrow X$ at the sum of suitable fundamental classes of all connected components of $Y$ is $c$.

Note that if any element $a$ of $A$ $a+a$ is zero, then we do not need orientations in the argument.

For related systematic topological theory on polyhedra and simplicical complexes and more general theory on algebraic topology, see \cite{hatcher} for example.  

As an extension of Fact \ref{fact:3}, we can show the following result. \cite{kitazawa3} concentrate on cases where preimages of regular values have only differentiable structures and no further structures.
\begin{Thm}
\label{thm:1}
Let $m>n$ be positive integers and let $M$ be a closed and smooth manifold.
Let $R$ be a PID and $A$ be a module over $R$.
Let $W_q$ be an $n$-dimensional polyhedron. Let $\tau \in H^{m-n}(M;A)$.
Let $f_q:M \rightarrow W_q$ be a pseudo quotient map of a class $\mathcal{C}$.
Assume that there exists a connected component of the preimage of a regular value where the pull-back of $\tau$ does not vanish.
In addition, assume the following two conditions in the case there exists an element $a$ of $A$ satisfying $2a \neq 0$.
\begin{enumerate}
\item $M$ is oriented.
\item For a triangulation giving $W_q$ the structure of the polyhedron in this situation, all simplices of degree $n$ in $W_q$ can be oriented satisfying the following conditions.
\begin{enumerate}
\item For each triangulable smooth map $f_{M,p}:M_p \rightarrow N_p$ in Definition \ref{def:2}, $M_p$ is taken as a closed, connected and oriented manifold and 
$N_p$ is taken as a connected and oriented manifold with no boundary. All simplices of degree $n$ in $W_{f_{M,p}}$ are oriented canonically by the map $\bar{f_{M,p}}:W_{f_{M,p}} \rightarrow N_p$.
\item Furthermore, we can take these maps and the oriented manifolds of the domains and the targets so that they restore the orientations of $M$ and all simplices of $W_q$ of degree $n$ via suitable diffeomorphisms ${\Phi}_p$ and PL homeomorphisms ${\phi}_{p^{\prime}}$ in Definition \ref{def:2}.
\end{enumerate}
\end{enumerate}
Then, the homology group $H_n(W_q;A)$ is not zero.
\end{Thm}
\begin{proof}

The proof is essentially due to the proofs of original theorems in \cite{hiratukasaeki}, \cite{hiratukasaeki2} and \cite{kitazawa3}. Let us construct a non-trivial top-dimensional cycle.

For any simplex $\sigma$ of $W_q$ of degree $n$, we can correspond a (an oriented) cobordism class uniquely by taking the class the preimage of a regular value belongs to. We can take the pull-back ${\tau}_{\sigma}$ of $\tau$ to the preimage and
the value at the fundamental class $f_{\sigma}$ (induced from suitable orientations of $M$ and simplices of $W_q$ of degree $n$ in the additional conditions if they are oriented). The resulting value is defined uniquely by the uniqueness of the cobordism class or the fact that for two different regular values in a simplex, the preimages are (oriented) cobordant in $M$. These arguments are are also in the proof of Lemma 3.1 of \cite{hiratukasaeki2} for example.

The value is zero if the simplex is not in the interior of $W_q$ or a face of the simplex is in the boundary and the sum ${\sum}_{\sigma} {\tau}_{\sigma}(f_{\sigma})\sigma$ for all simplices of $W_q$ of degree $n$ is a cycle. For the latter discussion, the key ingredients are a discussion in the original paper
and the fact that for each face ${\sigma}^{\prime}$ of degree $n-1$, the family $\{{\sigma}_{{\sigma}^{\prime},\lambda}\}_{\lambda}$ of all simplices of degree $n$ containing the face as a face in common and the disjoint union ${\sqcup}_{\lambda} {f_q}^{-1}(p_{\lambda})$ is a submanifold in $M$ and (resp. oriented) null-cobordant in $M$ where $p_{\lambda}$ denotes a regular value in the interior of ${\sigma}_{{\sigma}^{\prime},\lambda}$. Note that the additional conditions on orientations of $M$ and simplices of $W_q$ of degree $n$ yields the latter fact. These additional conditions mean that orientations of simplices of $W_q$ of degree $n$ are locally induced from a triangulable smooth map of the class $\mathcal{C}$.

The assumption that there exists a connected component of the preimage of a regular value where the pull-back of $\tau$ does not vanish implies that the obtained cycle is not zero. The cycle is also of degree $n=\dim W_q$. This completes the proof.

\end{proof}
\begin{Rem}
If $f_q:M \rightarrow W_q$ is the quotient map $q_f:M \rightarrow W_f$ of a triangulable smooth map $f:M \rightarrow N$ of the class $\mathcal{C}$ and $M$ and $N$ are oriented, then the conditions on the triangulation of $W_q=W_f$ are satisfied. 
\end{Rem}
\begin{Def}
\label{def:3}
In Theorem \ref{thm:1}, if the two conditions are satisfied in the case there exists an element $a$ of $A$ satisfying $2a \neq 0$, then we say that $W_q$ {\it can be oriented respecting $(f_q,\mathcal{C})$}.
\end{Def}

We show explicit applications.

Before that, we introduce fundamental notions on (smooth) embeddings.

We consider an embedding of a compact manifold $X$ into another manifold $Y$. $X$ is said to be {\it embedded properly} in $Y$ if the boundary of $X$ is mapped into the boundary of $Y$ and the interior of $X$ is mapped into the interior of $Y$ via the embedding. If a compact smooth manifold $X$ is embedded properly and smoothly into another smooth manifold $Y$, then we have a normal vector bundle ${\nu}_X$ over $X$ of dimension $\dim Y-\dim X$ and the restriction ${TY} {\mid}_X$ of $TY$ is regarded as a so-called {\it Whitney sum} of $TX$ and ${\nu}_X$.

Hereafter, we need notions of {\it spin structures} and {\it spin$^c$ structures} of compact and orientable smooth manifold.
We will avoid rigorous theory on these structures including several rigorous definitions.

First, a {\it spin} manifold is a compact and orientable smooth manifold whose {\it 2nd Stiefel-Whitney class} vanishes. This is a cohomology class of the manifold of degree $2$ whose coefficient ring is $\mathbb{Z}/2\mathbb{Z}$.
For a spin manifold $X$, {\it spin structures} are parametrized by the set of all cohomology classes of the manifold of degree $1$ whose coefficient ring is $\mathbb{Z}/2\mathbb{Z}$. They are defined as suitable equivalence classes of sequences of independent sections of the restriction of the tangent bundle or the Whitney sum of this and a trivial real vector bundle of dimension $1$ to the $1$-dimensional skeleton. Of course the simplicial structure is one giving the structure of the polyhedron of the differentiable manifold $X$. Furthermore, the length of each sequence and the dimension of the real vector bundle coincide.

We can define the spin structure {\it corresponding to a class of $H^1(X;\mathbb{Z}/2\mathbb{Z})$}.

We can generalize these notions for simplicial complexes and more generally CW complexes and real vector bundles over them. For a CW complex $X$ and an orientable real vector bundle whose dimension is greater than $0$, {\it spin structures} are defined as suitable equivalence classes of sequences of independent sections of the restrictions of the bundle or the Whitney sum of this and a real trivial vector bundle of dimension $1$ over $X$ to the $1$-dimensional skeletons. 

For a compact, oriented and smooth manifold $X$ embedded smoothly in a spin manifold $Y$ 
and the spin structure corresponding to $c_Y \in H^1(Y;\mathbb{Z}/2\mathbb{Z})$, assume that we can take a
trivial normal bundle of $X$ of dimension $\dim Y-\dim X$. In this situation, we can induce the spin structure corresponding to the pull-back of $c_Y$ on $X$.
For a compact and oriented manifold $X$ endowed with the spin structure corresponding to $c_X \in H^1(X;\mathbb{Z}/2\mathbb{Z})$ embedded smoothly in a spin manifold $Y$,
assume that we can take a
trivial normal bundle of $X$ of dimension $\dim Y-\dim X$.
If for each element in the kernel of the homomorphism from $H_1(X;\mathbb{Z}/2\mathbb{Z})$ into $H_1(Y;\mathbb{Z}/2\mathbb{Z})$ induced from the inclusion, the class $c_X$ vanishes and the homomorphism is surjective, then we can induce the spin structure in a unique way on $Y$. We can define the spin structure on $Y$ so that we can induce and restore $c_X$ by the pull-back.

A {\it spin$^c$} manifold is a compact, orientable and smooth manifold whose 2nd Stiefel-Whitney class is in the image of the canonical homomorphism from the integral 2nd cohomology group of the manifold into the 2nd cohomology group of the manifold whose coefficient ring is $\mathbb{Z}/2\mathbb{Z}$.
For a spin$^c$ manifold $X$, {\it spin$^c$ structures} are parametrized by the set of all cohomology classes of the manifold of degree $2$ whose coefficient ring is $\mathbb{Z}$. They are defined as suitable equivalence classes of pairs of $1$-dimensional complex vector bundles over the manifold $X$ and spin structures on the Whitney sums of the tangent bundle and the chosen complex vector bundles. We can define the spin$^c$ structure {\it corresponding to a class of $H^2(X;\mathbb{Z})$}.

For a compact, oriented and smooth manifold $X$ embedded smoothly in a spin$^c$ manifold $Y$ and the spin$^c$ structure corresponding to $c_Y \in H^2(Y;\mathbb{Z})$, assume that we can take a
trivial normal bundle of $X$ of dimension $\dim Y-\dim X$. In this situation, we can induce the spin$^c$ structure corresponding to the pull-back of $c_Y$ on $X$.
For a compact, oriented and smooth manifold $X$ endowed with the spin$^c$ structure corresponding to $c_X \in H^2(X;\mathbb{Z})$ embedded smoothly in a spin$^c$ manifold $Y$,
assume that we can take a
trivial normal bundle of $X$ of dimension $\dim Y-\dim X$.
Assume the following conditions in addition.
\begin{enumerate}
\item $H_1(X;\mathbb{Z})$ is free.
\item The homomorphism from $H_1(X;\mathbb{Z})$ into $H_1(Y;\mathbb{Z})$ induced from the inclusion is surjective.
\item For each element in the kernel of the homomorphism from $H_2(X;\mathbb{Z})$ into $H_2(Y;\mathbb{Z})$ induced from the inclusion, the class $c_X$ vanishes, and this homomorphism is surjective.
\end{enumerate}
We can induce the spin$^c$ structure in a unique way on $Y$. We can define the spin$^c$ structure on $Y$ so that we can induce and restore $c_X$ by the pull-back.

For more precise explanations on spin and spin$^c$ structures, see \cite{lawsonmichelson} for example. Through the proof of the following corollary, we present the spin structure of $S^1$ corresponding to the zero class and the spin$^c$ structure of a closed, connected and orientable surface corresponding to the zero class. We implicitly apply the argument before in this
and Theorems \ref{thm:2.1} and \ref{thm:2.2}.

A {\it $3$-dimensional handlebody} means a compact and connected smooth manifold represented as a boundary connected sum of finitely many copies of a manifold diffeomorphic to $S^2 \times D^1$.
\begin{Cor}
\label{cor:2}
Let $m>n$ be positive integers and let $M$ be an $m$-dimensional closed, orientable and smooth manifold.
Let $W_q$ be an $n$-dimensional pseudo quotient space of a class $\mathcal{C}$. Let $f_q:M \rightarrow W_q$ be a pseudo quotient map of a class $\mathcal{C}$.
Then the following two properties hold.
\begin{enumerate}
\item
\label{cor:2.1}
Let $m-n=1$ and $M$ be a manifold endowed with a spin structure. Assume that $H_n(W_q;\mathbb{Z}/2\mathbb{Z})$ is zero.
If we restrict the spin structure to each connected component of each preimage of each regular value, then the
obtained spin manifold diffeomorphic to the circle bounds a copy of $D^2$ as a spin manifold.
\item
\label{cor:2.2}
Let $m-n=2$ and $M$ be a manifold endowed with a spin$^c$ structure. Assume that $H_n(W_q;\mathbb{Z})$ is zero. 
Assume also that $W_q$ can be oriented respecting $(f_q,\mathcal{C})$. If we restrict the spin$^c$ structure
to each connected component of each preimage of each regular value, then the
obtained spin$^c$ manifold diffeomorphic to a closed and connected orientable manifold bounds a $3$-dimensional handlebody as a spin$^c$ manifold.
\end{enumerate}
\end{Cor}
\begin{proof}
In the proof of the first statement, it is essential that spin structures of a closed, spin and orientable manifold are parametrized by its 1st cohomology classes where the coefficient ring is $\mathbb{Z}/2\mathbb{Z}$. Note that the restriction of the structure to each closed submanifold with no boundary determines a spin structure of the submanifold. If the restriction of the structure corresponds to the zero class and the submanifold is a circle, then the circle bounds a copy of $D^2$ as a spin manifold. 
In the proof of the second statement, it is essential that spin$^c$ structures of a closed, orientable and spin$^c$ manifold are parametrized by its 2nd cohomology classes where the coefficient ring is $\mathbb{Z}$. If the restriction of the structure corresponds to the zero class and the submanifold is a $2$-dimensional sphere, then the sphere bounds a handlebody as a spin manifold. Note that $3$-dimensional handlebodies are
represented as boundary connected sums of finitely many copies of $S^1 \times D^2$ and that the 2nd homology groups are zero for the coefficient rings $\mathbb{Z}$ and $\mathbb{Z}/2\mathbb{Z}$.

By virtue of these facts together with the assumption on the top-dimensional homology group of the Reeb space and Theorem \ref{thm:1}, we obtain both results.

\end{proof}
We can argue further. 
We formulate a class of fold maps discussed in Example \ref{ex:1}. Around each singular value, the map has the form in Figure \ref{fig:1} if the preimage is connected. If the preimage is not connected, then around connected components except one, the map gives trivial smooth bundles whose fibers are diffeomorphic to unit spheres.
\begin{Def}
\label{def:4}
We say that a fold map $f:M \rightarrow N$ from an $m$-dimensional closed manifold $M$ into an $n$-dimensional manifold $N$ with no boundary satisfying the following conditions and properties is said to be {\it simple-standard-spherical}.
\begin{enumerate}
\item $m>n$.
\item The restriction of the quotient map $q_f$ to the singular set is injective. 
\item The index of each singular point is $0$ or $1$.
\item The preimages of regular values are disjoint unions of copies of a unit sphere.
\end{enumerate}
\end{Def}
A special generic map $f$ is also a simple-standard-spherical. Pseudo quotient maps in Example {\rm \ref{ex:1} (\ref{ex:1.1})} are pseudo quotient map of the class.

\begin{Thm}
\label{thm:2.1}
Let $m$ be a positive integer.
\begin{enumerate}
\item Let $M$ be an $m$-dimensional closed, smooth and spin manifold endowed with a spin structure.
Let $M$ admit a pseudo quotient map $f_q$ onto $W_{q}$ of the class $\mathcal{C}$ of simple-standard-spherical fold maps such that the differences of the dimensions of the manifolds of the domains and those of the targets are $1$ and $H_{m-1}(W_{q};\mathbb{Z}/2\mathbb{Z})$ vanish {\rm (}$m>1${\rm )}.

Then, $M$ bounds a compact, smooth and spin manifold so that the spin structure induced on $M$ and the original spin structure coincide.
\item
Let $M$ be an $m$-dimensional closed, smooth and spin$^c$ manifold endowed with a spin$^c$ structure.
Let $M$ admit a pseudo quotient map $f_q$ onto $W_{q}$ of the class $\mathcal{C}$ of simple-standard-spherical fold maps such that the differences of the dimensions of the manifolds of the domains and those of the targets are $2$ and $H_{m-2}(W_{q};\mathbb{Z})$ vanish {\rm (}$m>2${\rm )}.
Assume also that $W_q$ can be oriented respecting $(f_q,\mathcal{C})$.

Then, $M$ bounds a compact, smooth and spin$^c$ manifold $M_W$ so that the spin$^c$ structure induced on $M$ and the original spin$^c$ structure coincide.
\end{enumerate}
\end{Thm}
\begin{proof}
In both cases, we will show that the ($m+1$)-dimensional compact and smooth manifold $M_W$ in Example \ref{ex:1} is a desired manifold. We set the composition of the map $r_{M,P}$ producing the collapsing with a PL homeomorphism $r_P$ onto $P:=W_q$. This give a simple homotopy equivalence. We fix a triangulation of $P=W_q$.
By several fundamental explanations on spin and spin$^c$ structures before together with Corollary \ref{cor:2}, on the preimage $M_{{\rm W},N(S_1)} \subset M_W$ of a small regular neighborhood $N(S_1)$ of the $1$-skeleton $S_1 \subset W_{q}$, which is an (m+1)-dimensional compact and smooth submanifold of $M_W$, we can give the spin or spin$^c$ structure whose restriction to $M \bigcap M_{{\rm W},N(S_1)}$ agrees with the original spin or spin$^c$ structure induced from the given structure on $M$. On the preimage $M_{{\rm W},N(S_2)} \subset M_W$ of a small regular neighborhood $N(S_2)$ of the $2$-skeleton $S_2 \subset W_{q}$ satisfying $S_2 \supset S_1$, which is an (m+1)-dimensional compact and smooth submanifold of $M_W$, we can give the spin or spin$^c$ structure whose restriction to $M \bigcap M_{{\rm W},N(S_2)}$ agrees with the original spin or spin$^c$ structure induced from the given structure on $M$. 

By the topological structures of the manifolds and the maps together with fundamental theory on spin and spin$^c$ structures, we can induce the desired spin or spin$^c$ structure on $M_W$ respecting the previous spin or spin$^c$ structure on $M_{{\rm W},N(S_1)}$, that on $M_{{\rm W},N(S_2)}$, and the original spin or spin$^c$ structure on $M$.
\end{proof}

We consider stable fold maps into the plane from $3$-dimensional closed and orientable manifolds with no {II}$^3$ type fiber in Example \ref{ex:1} (\ref{ex:1.2}) or $4$-dimensional closed and orientable manifolds.
As in Example \ref{ex:1} (\ref{ex:1.2}), the class of stable fold maps into the plane from $4$-dimensional closed manifolds (which may not be orientable) is studied in \cite{saekisuzuoka} for example.
   
For a map of these classes, remove points and the interiors of suitable small regular neighborhoods in the manifold of the target such that for local stable fold maps in
Definition \ref{def:2}, we need connected components of preimages containing two singular points. The intersection of the resulting complementary set and the regular value set is, as a result, restricted to a manifold $N_R$ PL homeomorphic to a compact surface. The intersection of the resulting complementary set and the singular value set is, as a result, restricted to a manifold $N_S$ PL homeomorphic to a disjoint union of circles or closed intervals.

As in the last of Example \ref{ex:1} (\ref{ex:1.2}), in such a case, we can naturally define a PL map $r_{M,P}$ with nice topological properties on the constructed $4$ or $5$-dimensional manifold $M_W$ inducing collapsing to a suitable subpolyhedron $P^{\prime}$ admitting a PL map $r_P$ onto the polyhedron $P$ of the target. We consider the composition $r_P \circ r_{M,P}$. The preimage of the small regular neighborhood of the intersection $N_R$ of the resulting complementary set and the regular value set 
before is a submanifold of the constructed compact manifold $M_W$. It is regarded as the total space of a bundle over the submanifold in the regular value set whose fiber is diffeomorphic to the $2$ or $3$-dimensional unit disk. The preimage of the small regular neighborhood of the intersection $N_S$ of the resulting complementary set and the singular value set 
before is a submanifold of the constructed compact manifold $M_W$. It is regarded as the total space of a bundle over the submanifold in the singular value set whose fiber is diffeomorphic to the $3$ or $4$-dimensional unit disk. If we argue in the PL or the piecewise smooth category, then the bundles are PL bundles. We can also argue in the smooth category in this case for example. In this case, the bundles are also smooth.

Furthermore, the preimages of the small regular neighborhoods of points whose preimages have connected components containing (at least) two singular points (as presented in FIGURES \ref{fig:1} and \ref{fig:2}) in the singular value sets of the original pseudo quotient maps before are compact smooth submanifolds of the constructed compact manifold $M_W$. These manifolds are diffeomorphic to the $4$ or $5$-dimensional unit disk.

The maps $r_P$ and $r_{M,P}$ can be taken so that the composition gives a simple homotopy equivalence. Such arguments are in the proof of the main theorem of \cite{saekisuzuoka} for example.

By virtue of these fundamental discussions, the proof of Theorem \ref{thm:2.1}, and Theorem \ref{thm:1}, for example, we also have the following theorem.

\begin{Thm}
\label{thm:2.2}
\begin{enumerate}
\item Let $\mathcal{C}$ be the class of all stable fold maps into the plane from $3$-dimensional closed and orientable manifolds with no ${II}^3$ type fibers. Let a $3$-dimensional closed, connected and spin manifold $M$ admit a pseudo quotient map $f_q$ onto $W_{q}$ of the class $\mathcal{C}$ and $H_{2}(W_{q};\mathbb{Z}/2\mathbb{Z})$ vanish. Then, $M$ bounds a compact, smooth and spin manifold $M_W$ as a spin manifold.
\item Let $\mathcal{C}$ be the class of all stable fold maps into the plane from $4$-dimensional closed manifolds which may not be orientable into the plane such that the preimages of regular values are disjoint unions of copies of the $2$-dimensional unit sphere. Let a $4$-dimensional closed, connected and spin$^c$ manifold $M$ admit a pseudo quotient map $f_q$ onto $W_{q}$ of the class $\mathcal{C}$ and
$H_{2}(W_{q};\mathbb{Z})$ vanish. 

Assume also that $W_{q}$ can be oriented respecting $(f_q,\mathcal{C})$.

Then, $M$ bounds a compact, smooth and spin$^c$ manifold $M_W$ as a spin$^c$ manifold. 
\end{enumerate}
\end{Thm}

We introduce a module defined from cobordisms in the {\it framed category}. For precise explanations of terminologies and fundamental properties of {\it framed manifolds} and the {\it framed category}, see \cite{milnor} for example.

Let $X$ be a compact and smooth manifold. Let $k<{\dim} X$ be a non-negative integer. For a $k$-dimensional compact and properly and smoothly embedded submanifold $S$ in $X$, let there
exist the following two objects.
\begin{enumerate}
\item A normal bundle ${\nu}_S$.
\item A sequence $\mathcal{S}$ of ${\dim} X-k$ independent sections of the bundle ${\nu}_S$ such that on the boundary $\partial S$, the values are in ${{\nu}_S} {\mid}_{\partial S} \bigcap {T\partial X} {\mid}_{\partial S}$.
\end{enumerate}
Then $(S,{\nu}_S,\mathcal{S})$ is said to be a {\it framed manifold} in $X$.
Two framed manifolds in a closed manifold $X$ is {\it cobordant} as framed manifolds in $X$ if there exists a framed manifold of $X \times [0,1]$ such that the framed manifold canonically obtained by the restriction of the manifold $X \times [0,1]$ to the boundary $X \times \{0,1\}$ is the disjoint union of the original two manifolds where one is canonically defined as a framed manifold in $X =X \times \{0\}$ and the other is canonically defined as a framed manifold in $X=X \times \{1\}$. Note that we adopt the identification of $x \in X$ with $(x,0)$ and $(x,1)$ for $x \in X$.  

Let $R$ be a PID.
Let $M$ be an $m$-dimensional closed, connected and smooth manifold. Let $k<m$ be a non-negative integer.
We define an $R$-module $\mathcal{F}_k(R,M)$ as a free module generated by all $k$-dimensional framed closed and connected manifolds in $M$ such that distinct $k$-dimensional framed closed and connected manifolds are mutually independent in the module. 
\begin{Def}
\label{def:5}
A submodule $A \subset \mathcal{F}_k(R,M)$ over $R$ is said to be {\it compatible} with the free module if for any framed compact ($k+1$)-dimensional manifold $F_0$ in $M \times [0,1]$, the sum of the two elements $F_1$ and $F_2$ obtained in the following way is an element of $A$. 
\begin{enumerate}
\item Let $F_{0,1}$ and $F_{0,2}$ be framed manifolds in $M \times \{0\}$ and $M \times \{1\}$ respectively and they are obtained by restricting $F_0$ and $M \times [0,1]$ to the boundary.
\item By identifying as $M =M \times \{0\}=M \times \{1\}$ via the identification of the three elements $(x,0)$, $(x,1)$ and $x$, We have a framed manifolds $F_1$ and $F_2$ in $M$.
\end{enumerate}
\end{Def}
\begin{Rem}
(Some) elements in $A$ in Definition \ref{def:5} are for {\it cobordism} relations of framed manifolds. 
In the case where the relation $m-k \geq \frac{1}{2}m+1$ holds,  $\mathcal{F}_k(\mathbb{Z},M)/A$ can be a well-known commutative group called degree $m-k$ {\it cohomotopy group} of $M$. The sum of this commutative group is defined by taking the disjoint union of corresponding framed manifolds after perturbing the (framed) manifolds via suitable isotopies if we need.
\end{Rem}

\begin{Thm}
\label{thm:3}
Let $m>n$ be positive integers. Let $M$ be an $m$-dimensional closed and smooth manifold.
Let $R$ be a PID and let $A$ be an $R$-module compatible with the free module $\mathcal{F}_{m-n}(R,M)$.
Let $N$ be an $n$-dimensional oriented manifold with no boundary.
Let $f:M \rightarrow N$ be a triangulable smooth map.

The preimage of a regular value can be regarded as a framed submanifold in a natural way due to the theory of \cite{milnor}. 

Furthermore, suppose that the following two additional conditions hold.
\begin{enumerate}
\item There exists a connected component $F$ of the preimage of a regular value, which is regarded as a framed manifold in the way just before.
\item The value of the natural quotient map from $\mathcal{F}_{m-n}(R,M)$ onto $\mathcal{F}_{m-n}(R,M)/A$ {\rm (}, defined by virtue of fundamental properties of framed manifolds{\rm )} at the framed manifold $F$ is not zero.
\end{enumerate}
Then $H_n(W_f;\mathcal{F}_{m-n}(R,M)/A)$ is not zero.
\end{Thm}
\begin{proof}
As the proof of Theorem \ref{thm:1}, for any simplex $\sigma$ of $W_f$ of degree $n$, we can correspond an element in ${\mathcal{F}}_{\sigma} \in \mathcal{F}_k(R,M)/A$ uniquely by taking the element uniquely defined from the preimage of a regular value. The preimage is a submanifold with no boundary and also a framed manifold of $M$. The uniqueness of the element or the class is due to the following reasons.
\begin{itemize}
\item $N$ is oriented. We can choose a basis of the tangent vector space at any point of $N$ compatible with the orientation of $N$. 
For a regular value $p \in N$ of $f$, we choose such a basis $b_p$. The preimage $F_p:=f^{-1}(p)$ is regarded as a framed manifold with a normal bundle ${\nu}_{F_p}$ and we have suitable $\dim N$ independent sections. We can do so that the total space of the normal bundle is regarded as the preimage of the interior of a small copy of the $\dim N$-dimensional unit disk embedded smoothly in $N$ and containing $p$ in the inteiror. We can find the normal bundle and the mutually independent sections so that the value obtained as the value of the composition of each of the $\dim N$ independent sections before with the differential of the restriction of $f$ to ${\nu}_{F_p}$ is uniquely defined and the corresponding element of the basis $b_p$. This is due to related arguments in \cite{milnor}.
\item For two different regular values, preimages are cobordant as framed manifolds in $M$ obtained in the manner just before. We can see this from the fundamental discussion presented in \cite{milnor} and Lemma 3.1 of \cite{hiratukasaeki2}, for example. 
\end{itemize}

The first reason presents the fact that the preimage of any regular value can be regarded as a framed submanifold in a natural way due to the theory of \cite{milnor}. 

We can know that the sum ${\sum}_{\sigma} {\mathcal{F}}_{\sigma}\sigma$ for all simplices of $W_f$ of degree $\dim Y$ is a cycle and that the homology class the cycle represents does not vanish by the two additional conditions.

This completes the proof.
\end{proof}

\begin{Ex}
In \cite{kitazawa} and \cite{kitazawa2} for example, the total space of a smooth bundle whose fiber is diffeomorphic to the $k$-dimensional unit sphere for $k \geq 1$ over the $2$-dimensional unit sphere is shown to admit a fold map into the plane satisfying the following properties.
\begin{enumerate}
\item The restriction to the singular set is an embedding.
\item The singular set is the disjoint union of two circles.
\item There exists exactly one connected component of the regular value set which is diffeomorphic to the interior of the $2$-dimensional unit disk. The preimage of a regular value in this connected component is diffeomorphic to the disjoint union of two copies of the $k$-dimensional unit sphere. Note that the regular value set consists of three connected components.
\end{enumerate}
This is also a simple-standard-spherical fold map of Definition \ref{def:4}. The Reeb space is simple homotopy equivalent to the $2$-dimensional unit sphere. Set $k=1$ and set the manifold of the domain as a copy of the $3$-dimensional unit sphere or the total space of a so-called {\it Hopf-fibration} over the $2$-dimensional sphere with the canonical orientation. 
A Hopf-fibration is a smooth bundle over $S^2$ whose fiber is a circle. In this situation, each connected component is a circle. It is also regarded as a fiber of the bundle and is not null-cobordant as a canonically induced framed manifold.    
\end{Ex}
\section{Acknowledgement and data.}
\label{sec:5}
The author is a member of JSPS KAKENHI Grant Number JP17H06128 "Innovative research of geometric topology and singularities of differentiable mappings" (Principal Investigator: Osamu Saeki) and supported by this project. We also declare that data supporting our present study are all in the present paper. 

\end{document}